\documentclass[12pt,a4paper]{article}
\usepackage[utf8]{inputenc}
\usepackage{amsmath}
\usepackage{amssymb}
\usepackage{esint}

\makeatletter

\newcommand{\lyxaddress}[1]{
\par {\raggedright #1
\vspace{1.4em}
\noindent\par}
}

\newcommand{\mathsym}[1]{{}}

\usepackage{amsthm}
\usepackage{mathrsfs}

\newcommand{\sM}{\mathscr{M}}
\newcommand{\sD}{\mathscr{D}}
\newcommand{\sN}{\mathscr{N}}
\newcommand{\sS}{\mathscr{S}}
\newcommand{\sH}{\mathscr{H}}
\newcommand{\sB}{\mathscr{B}}

\newtheorem{theorem}{Theorem}
\newtheorem{proposition}[theorem]{Proposition}
\newtheorem{lemma}[theorem]{Lemma}

\theoremstyle{remark}
\newtheorem{remark}[theorem]{Remark}
\newtheorem*{question*}{QUESTION}
\newtheorem*{problem*}{Problem}
\newtheorem*{remark*}{Remark}
\newtheorem*{definition*}{Definition}
\newtheorem*{notation*}{Notation}

\newcommand{\Dom}{\mathop\mathrm{Dom}\nolimits}
\newcommand{\Ran}{\mathop\mathrm{Ran}\nolimits}
\newcommand{\Ker}{\mathop\mathrm{Ker}\nolimits}
\newcommand{\spec}{\mathop\mathrm{spec}\nolimits}
\renewcommand{\Re}{\mathop\mathrm{Re}\nolimits}
\renewcommand{\Im}{\mathop\mathrm{Im}\nolimits}

\makeatother

\begin{document}

\title{The Hahn-Exton $q$-Bessel function as the characteristic function
of a Jacobi matrix}

\author{F.~\v{S}tampach$^{1}$, P.~\v{S}\v{t}ov\'\i\v{c}ek$^{2}$}

\date{{}}

\maketitle

\lyxaddress{$^{1}$Department of Applied Mathematics, Faculty of Information
Technology, Czech Technical University in~Prague, Kolejn\'\i~2,
160~00 Praha, Czech Republic}

\lyxaddress{$^{2}$Department of Mathematics, Faculty of Nuclear Science, Czech
Technical University in Prague, Trojanova 13, 12000 Praha, Czech Republic}
\begin{abstract}
\noindent A family $\mathcal{T}^{(\nu)}$, $\nu\in\mathbb{R}$, of
semiinfinite positive Jacobi matrices is introduced with matrix entries
taken from the Hahn-Exton $q$-difference equation. The corresponding
matrix operators defined on the linear hull of the canonical basis
in $\ell^{2}(\mathbb{Z}_{+})$ are essentially self-adjoint for $|\nu|\geq1$
and have deficiency indices $(1,1)$ for $|\nu|<1$. A convenient
description of all self-adjoint extensions is obtained and the spectral
problem is analyzed in detail. The spectrum is discrete and the characteristic
equation on eigenvalues is derived explicitly in all cases. Particularly,
the Hahn-Exton $q$-Bessel function $J_{\nu}(z;q)$ serves as the
characteristic function of the Friedrichs extension. As a direct application
one can reproduce, in an alternative way, some basic results about
the $q$-Bessel function due to Koelink and Swarttouw.
\end{abstract}
\vskip\baselineskip\noindent\emph{Keywords}: Jacobi matrix, Hahn-Exton
$q$-Bessel function, self-adjoint extension, spectral problem

\vskip0.5\baselineskip\noindent\emph{2010 Mathematical Subject
Classification}: 47B36, 33D45, 47A10, 39A70

\section{Introduction}

There exist three commonly used $q$-analogues of the Bessel function
$J_{\nu}(z)$. Two of them were introduced by Jackson in the beginning
of the 20th century and are mutually closely related, see \cite{GasperRahman}
for a basic overview and original references. Here we shall be concerned
with the third analogue usually named after Hahn and Exton. Its most
important features like properties of the zeros and the associated
Lommel polynomials including orthogonality relations were studied
not so long ago \cite{KoelinkSwarttouw,KoelinkVanAssche,Koelink}.
The Hahn-Exton $q$-Bessel function is defined as follows
\begin{equation}
J_{\nu}(z;q)\equiv J_{\nu}^{(3)}(z;q)=\frac{(q^{\nu+1};q){}_{\infty}}{(q;q)_{\infty}}\, z^{\nu}\,_{1}\phi_{1}(0;q^{\nu+1};q,qz^{2}).\label{eq:def_q-Bessel}
\end{equation}
Here $\,_{r}\phi_{s}(a_{1},\ldots,a_{r};b_{1},\ldots,b_{s};q,z)$
stands for the basic hypergeometric series (see, for instance, \cite{GasperRahman}).
It is of importance that $J_{\nu}(z;q)$ obeys the Hahn-Exton $q$-Bessel
difference equation
\begin{equation}
J_{\nu}(qz;q)+q^{-\nu/2}(qz^{2}-1-q^{\nu})J_{\nu}(q^{1/2}z;q)+J_{\nu}(z;q)=0.\label{eq:Hahn-Exton_diffeq}
\end{equation}

Using the coefficients from (\ref{eq:Hahn-Exton_diffeq}) one can
introduce a two-parameter family of real symmetric Jacobi matrices

\begin{equation}
\mathcal{T}\equiv\mathcal{T}^{(\nu)}=\begin{pmatrix}\beta_{0} & \alpha_{0}\\
\alpha_{0} & \beta_{1} & \alpha_{1}\\
 & \alpha_{1} & \beta_{2} & \alpha_{2}\\
 &  & \ddots & \ddots & \ddots
\end{pmatrix}\label{eq:JacobiT}
\end{equation}
depending on $\nu\in\mathbb{R}$ and also on $q$, $0<q<1$. But $q$
is treated below as having been fixed and is not indicated explicitly
in most cases. Matrix entries are supposed to be indexed by $m,n=0,1,2,\ldots$.
More formally, we put $\mathcal{T}_{n,n}=\beta_{n}$, $\mathcal{T}_{n,n+1}=\mathcal{T}_{n+1,n}=\alpha_{n}$
and $\mathcal{T}_{m,n}=0$ otherwise, where
\begin{equation}
\alpha_{n}\equiv\alpha_{n}^{(\nu)}=-q^{-n+(\nu-1)/2},\text{ }\beta_{n}\equiv\beta_{n}^{(\nu)}=(1+q^{\nu})\, q^{-n},\text{ }n\in\mathbb{Z}_{+}.\label{eq:alpha_beta}
\end{equation}
In order to keep notations simple we will also suppress the superscript
$(\nu)$ provided this cannot lead to misunderstanding.

Our main goal in this paper is to provide a detailed analysis of those
operators $T$ in $\ell^{2}\equiv\ell^{2}(\mathbb{Z}_{+})$ (with
$\mathbb{Z}_{+}$ standing for nonnegative integers) whose matrix
in the canonical basis equals $\mathcal{T}$. This example has that
interesting feature that it exhibits a transition between the indeterminate
and determinate cases depending on $\nu$. In more detail, denote
by $\mathbb{C}^{\infty}$ the linear space of all complex sequences
indexed by $\mathbb{Z}_{+}$ and by $\sD$ the subspace of those sequences
having at most finitely many nonvanishing entries. One may also say
that $\sD$ is the linear hull of the canonical basis in $\ell^{2}$.
It turns out that the matrix operator induced by $\mathcal{T}$ on
the domain $\sD$ is essentially self-adjoint in $\ell^{2}$ if and
only if $|\nu|\geq1$. For $|\nu|<1$ there exists a one-parameter
family of self-adjoint extensions.

Another interesting point is a close relationship between the spectral
data for these operators $T$ and the Hahn-Exton $q$-Bessel function.
It turns out that, for an appropriate (Friedrichs) self-adjoint extension,
$J_{\nu}(q^{-1/2}\sqrt{x};q)$ serves as the characteristic function
of $T$ in the sense that its zero set on $\mathbb{R}_{+}$ exactly
coincides with the spectrum of $T$. There also exists an explicit
formula for corresponding eigenvectors. Moreover, $T^{-1}$ can be
shown to be compact. This makes it possible to reproduce, in a quite
straightforward but alternative way, some results originally derived
in \cite{KoelinkSwarttouw,Koelink}.

Finally we remark that recently we have constructed, in \cite{StampachStovicek13a,StampachStovicek13b},
a number of examples of Jacobi operators with discrete spectra and
characteristic functions explicitly expressed in terms of special
functions, a good deal of them comprising various combinations of
$q$-Bessel functions. That construction confines, however, only to
a class of Jacobi matrices characterized by a convergence condition
imposed on the matrix entries. For this condition is readily seen
to be violated in the case of $\mathcal{T}$, as defined in (\ref{eq:JacobiT})
and (\ref{eq:alpha_beta}), in the present paper we have to undertake
another approach whose essential part is a careful asymptotic analysis
of formal eigenvectors of $\mathcal{T}$.

\section{Self-adjoint operators induced by $\mathcal{T}$}

\subsection{A $\ast$-algebra of semiinfinite matrices \label{subsec:Mfin}}

Denote by $\sM_{\text{fin}}$ the set of all semiinfinite matrices
indexed by $\mathbb{Z}_{+}\times\mathbb{Z}_{+}$ such that each row
and column of a matrix has only finitely many nonzero entries. For
instance, $\sM_{\text{fin}}$ comprises all band matrices and so all
finite-order difference operators. Notice that $\sM_{\text{fin}}$
is naturally endowed with the structure of a $\ast$-algebra, matrices
from $\sM_{\text{fin}}$ act linearly on $\mathbb{C}^{\infty}$ and
$\sD$ is $\sM_{\text{fin}}$-invariant.

Choose $\mathcal{A}\in\sM_{\text{fin}}$ and let $\mathcal{A}^{\text{H}}$
stand for its Hermitian adjoint. Let us introduce, in a fully standard
manner, operators $\dot{A}$, $A_{\text{min}}$ and $A_{\text{max}}$
on $\mathbb{\ell}^{2}$, all of them being restrictions of $\mathcal{A}$
to appropriate domains. Namely, $\dot{A}$ is the restriction $\mathcal{A}\big|_{{\displaystyle \sD}}$,
$A_{\text{min}}$ is the closure of $\dot{A}$ and
\[
\Dom A_{\text{max}}=\{f\in\ell^{2};\,\mathcal{A}f\in\ell^{2}\}.
\]
Clearly, $\dot{A}\subset A_{\text{max}}$. Straightforward arguments
based just on systematic application of definitions show that
\[
(\dot{A})^{\ast}=(A_{\text{min}})^{\ast}=A_{\text{max}}^{\text{H}},\ (A_{\text{max}})^{\ast}=A_{\text{min}}^{\text{H}}.
\]
Hence $A_{\text{max}}$ is closed and $A_{\text{min}}\subset A_{\text{max}}$.

\begin{lemma} \label{prop:Amin_Amax} Suppose $p,w\in\mathbb{C}$
and let $\mathcal{A}\in\sM_{\text{fin}}$ be defined by
\begin{equation}
\mathcal{A}_{n,n}=p^{n},\ \mathcal{A}_{n+1,n}=-wp^{n+1}\ \ \text{for all}\ n\in\mathbb{Z}_{+},\ \mathcal{A}_{m,n}=0\ \text{otherwise}.\label{eq:Acal}
\end{equation}
Then $A_{\text{min}}\neq A_{\text{max}}$ if and only if $1/|p|<|w|<1$,
and in that case
\[
\Dom A_{\text{min}}=\{f\in\Dom A_{\text{max}}\,;\,\lim_{n\to\infty}w^{-n}f_{n}=0\}.
\]
\end{lemma}

\begin{proof} Choose arbitrary $f\in\Dom A_{\text{max}}$. Then $f\in\Dom A_{\text{min}}$
iff
\begin{equation}
\forall g\in\Dom A_{\text{max}}^{\text{H}},\ 0=\langle\mathcal{A}^{\text{H}}g,f\rangle-\langle g,\mathcal{A}f\rangle=-\lim_{n\to\infty}\mathcal{A}_{n,n}\,\overline{g_{n}}f_{n}.\label{eq:der_DomA_min}
\end{equation}
Since both $f$ and $g$ in (\ref{eq:der_DomA_min}) are supposed
to belong to $\ell^{2}$ this condition is obviously fulfilled if
$|p|\leq1$. Furthermore, the situation becomes fully transparent
for $w=0$. In that case the sequences $\{p^{n}g_{n}\}$ and $\{p^{n}f_{n}\}$
are square summable and (\ref{eq:der_DomA_min}) is always fulfilled.
In the remainder of the proof we assume that $|p|>1$ and $w\neq0$.

Consider first the case when $|w|\geq1$. Relation $\mathcal{A}f=h$
can readily be inverted even in $\mathbb{C}^{\infty}$ and one finds
that
\[
p^{n}f_{n}=\sum_{k=0}^{n}(pw)^{k}h_{n-k}=(pw)^{n}\sum_{k=0}^{n}(pw)^{-k}h_{k},\ \forall n.
\]
Denote temporarily by $\tilde{h}$ the sequence with $\tilde{h}_{n}=(\overline{p}\overline{w})^{-n}$.
It is square summable since, by our assumptions, $|pw|>1$. For $f\in\Dom A_{\text{max}}$
one has $h\in\ell^{2}$ and
\[
f_{n}=w^{n}(\langle\tilde{h},h\rangle-\zeta_{n})\,\ \text{where}\ \zeta_{n}=\sum_{k=n+1}^{\infty}(pw)^{-k}h_{k}.
\]
Assumption $f\in\ell^{2}$ clearly implies $\langle\tilde{h},h\rangle=0$
and then, by the Schwarz inequality,
\[
|\mathcal{A}_{n,n}f_{n}|\leq\|h\|/\sqrt{|pw|-1}\,,\ \forall n.
\]
Whence $\mathcal{A}_{n,n}\,\overline{g_{n}}f_{n}\to0$ as $n\to\infty$
for all $g\in\Dom A_{\text{max}}^{\text{H}}$ and so $f\in\Dom A_{\text{min}}$.

Suppose now that $|w|<1$. If $\mathcal{A}^{\text{H}}g=h$ in $\mathbb{C}^{\infty}$
and $h$ is bounded then, as an easy computation shows,
\begin{equation}
(\overline{p})^{n}\, g_{n}=\gamma(\overline{w})^{-n}+\sum_{k=0}^{\infty}(\overline{w})^{k}h_{n+k}\label{eq:aux_const_gamma}
\end{equation}
for all $n$ and some constant $\gamma$. Observe that, by the Schwarz
inequality,
\begin{equation}
\left|\sum_{k=0}^{\infty}(\overline{w})^{k}h_{n+k}\right|\leq\frac{1}{\sqrt{1-|w|^{2}}}\left(\sum_{k=n}^{\infty}|h_{k}|^{2}\right)^{\!1/2},\label{eq:Schwarz_wk_h}
\end{equation}
and this expression tends to zero as $n$ tends to infinity provided
$h\in\ell^{2}$.

In the case when $|pw|\leq1$ the property $g\in\ell^{2}$ and $\mathcal{A}^{\text{H}}g=h\in\ell^{2}$
implies that the constant $\gamma$ in (\ref{eq:aux_const_gamma})
is zero, and from (\ref{eq:Schwarz_wk_h}) one infers that $\mathcal{A}_{n,n}\,\overline{g_{n}}\to0$
as $n\to\infty$. Thus one finds condition (\ref{eq:der_DomA_min})
to be always fulfilled meaning that $f\in\Dom A_{\text{min}}$.

If $|pw|>1$ then the sequence $g$ defined in (\ref{eq:aux_const_gamma})
is square summable whatever $\gamma\in\mathbb{C}$ and $h\in\ell^{2}$
are. Condition (\ref{eq:der_DomA_min}) is automatically fulfilled,
however, for $\gamma=0$. Hence (\ref{eq:der_DomA_min}) can be reduced
to the single nontrivial case when we choose $\tilde{g}\in\Dom A_{\text{max}}^{\text{H}}$
with $\tilde{g}_{n}=(\overline{p}\overline{w})^{-n}$. Then $\mathcal{A}^{\text{H}}\tilde{g}=0$
and condition $\langle\tilde{g},\mathcal{A}f\rangle=0$ means that
$w^{-n}f_{n}\to0$ as $n\to\infty$. It remains to show that there
exists $f\in\Dom A_{\text{max}}$ not having this property. However
the sequence $\tilde{f}$, with $\tilde{f}_{n}=w^{n}$, does the job
since $\mathcal{A}\tilde{f}=(1,0,0,\ldots)\in\ell^{2}$. \end{proof}

\subsection{Associated orthogonal polynomials, self-adjoint extensions}

The tridiagonal matrix $\mathcal{T}$ defined in (\ref{eq:JacobiT}),
(\ref{eq:alpha_beta}) belongs to $\sM_{\text{fin}}$. With $\mathcal{T}$
there is associated a sequence of monic orthogonal polynomials \cite{Chihara},
called $\{P_{n}(x)\equiv P_{n}^{(\nu)}(x)\}$ and defined by the recurrence
\begin{equation}
P_{n}(x)=(x-\beta_{n-1})P_{n-1}(x)-\alpha_{n-2}^{\,\,2}\, P_{n-2}(x),\text{ }n\geq1,\label{eq:Pn_def}
\end{equation}
with $P_{-1}(x)=0$, $P_{0}(x)=1$. Put
\begin{equation}
\hat{P}_{n}(x)\equiv\hat{P}_{n}^{(\nu)}(x)=(-1)^{n}q^{n(n-\nu)/2}P_{n}(x).\label{eq:Phat_def}
\end{equation}
Then $(\hat{P}_{0}(x),\hat{P}_{1}(x),\hat{P}_{2}(x),\ldots)$ is a
formal eigenvector of $\mathcal{T}$ ($\equiv$~an eigenvector of
$\mathcal{T}$ in $\mathbb{C}^{\infty}$), i.e.
\begin{equation}
(\beta_{0}-x)\hat{P}_{0}(x)+\alpha_{0}\hat{P}_{1}(x)=0,\text{ }\alpha_{n-1}\hat{P}_{n-1}(x)+(\beta_{n}-x)\hat{P}_{n}(x)+\alpha_{n}\hat{P}_{n+1}(x)=0\text{ }\text{for}\text{ }n\geq1.\label{eq:recurr_orig}
\end{equation}

Observe that $\mathcal{T}^{(-\nu)}=q^{-\nu}\,\mathcal{T}^{(\nu)}$.
Since we are primarily interested in spectral properties of $\mathcal{T}^{(\nu)}$
in the Hilbert space $\ell^{2}$ we may restrict ourselves, without
loss of generality, to nonnegative values of the parameter $\nu$.
The value $\nu=0$ turns out to be somewhat special and will be discussed
separately later on, in Subsection~\ref{subsec:nu_eq_0}. Thus, if
not stated otherwise, we assume from now on that $\nu>0$.

Given $\mathcal{T}\in\sM_{\text{fin}}$ we again introduce the operators
$\dot{T}$, $T_{\text{min}}$, $T_{\text{max}}$ as explained in Subsection~\ref{subsec:Mfin}.
Notice that 
\[
\beta_{n}=q^{(\nu-1)/2}|\alpha_{n-1}|+q^{-(\nu-1)/2}|\alpha_{n}|.
\]
It follows at once that the operators $\dot{T}$ and consequently
$T_{\min}$ are positive. In fact, for any real sequence $\left\{ f_{n}\right\} \in\sD$
one has
\[
\sum_{m=0}^{\infty}\sum_{n=0}^{\infty}\mathcal{T}_{m,n}f_{m}f_{n}=|\alpha_{-1}|q^{(\nu-1)/2}\, f_{0}^{\,2}+\sum_{n=1}^{\infty}|\alpha_{n-1}|(q^{(\nu-1)/4}f_{n}-q^{-(\nu-1)/4}f_{n-1})^{2}\geq0.
\]
This is equivalent to the factorization $\mathcal{T=}\mathcal{A}^{\text{H}}\mathcal{A}$
where the matrix $\mathcal{A}\equiv\mathcal{A}^{(\nu)}\in\sM_{\text{fin}}$
is defined by the prescription: $\forall f\in\mathbb{C}^{\infty}$,
\[
(\mathcal{A}f)_{0}=|\alpha_{-1}|^{1/2}q^{(\nu-1)/4}\, f_{0},\ (\mathcal{A}f)_{n}=|\alpha_{n-1}|^{1/2}(q^{(\nu-1)/4}f_{n}-q^{-(\nu-1)/4}f_{n-1})\,\ \text{for}\ n\geq1.
\]
That is, $\forall n\geq0$,
\begin{equation}
\mathcal{A}_{n,n}=|\alpha_{n-1}|^{1/2}q^{(\nu-1)/4}=q^{-(n-\nu)/2},\ \mathcal{A}_{n+1,n}=-|\alpha_{n}|^{1/2}q^{-(\nu-1)/4}=-q^{-n/2}\,,\label{eq:A_from_decompT}
\end{equation}
and $\mathcal{A}_{m,n}=0$ otherwise.

Thus $\mathcal{T}$ induces a positive form on the domain $\sD$ with
values $\langle f,\mathcal{T}f\rangle=\|\mathcal{A}f\|^{2}$, $\forall f\in\sD$.
Let us call $\mathfrak{t}$ its closure. Then $\Dom\mathfrak{t}=\Dom A_{\text{min}}$
and $\mathfrak{t}(x)=\|A_{\text{min}}\, x\|^{2}$, $\forall x\in\Dom\mathfrak{t}$.
The positive operator $T^{\text{F}}$ associated with $\mathfrak{t}$
according to the representation theorem is the Friedrichs extension
of the closed positive operator $T_{\text{min}}$. One has
\[
T^{\text{F}}=A_{\text{min}}^{\,\ast}A_{\text{min}}.
\]
It is known that $T^{\text{F}}$ has the smallest form-domain among
all self-adjoint extensions of $T_{\text{min}}$ and also that this
is the only self-adjoint extension of $T_{\text{min}}$ with its domain
contained in $\Dom\mathfrak{t}$, see \cite[Chapter~VI]{Kato}.

One can apply Lemma~\ref{prop:Amin_Amax}, with $p=q^{-1/2}$ and
$w=q^{(1-\nu)/2}$, to obtain an explicit description of the form-domain
of $T^{\text{F}}$. Using still $\mathcal{A}$ defined in (\ref{eq:A_from_decompT})
one has $\Dom\mathfrak{t}=\{f\in\ell^{2};\,\mathcal{A}f\in\ell^{2}\}$
for $\nu\geq1$ and
\begin{equation}
\Dom\mathfrak{t}=\{f\in\ell^{2};\,\mathcal{A}f\in\ell^{2}\ \text{and}\ \lim_{n\to\infty}q^{(\nu-1)n/2}f_{n}=0\}\label{eq:Dom_t}
\end{equation}
for $0<\nu<1$.

In \cite{BrownChristiansen} one finds a clear explicit description
of the domain of the Friedrichs extension of a positive Jacobi matrix
which can be applied to our case. To this end, consider the homogeneous
three-term recurrence equation
\begin{equation}
\alpha_{n}Q_{n+1}+\beta_{n}Q_{n}+\alpha_{n-1}Q_{n-1}=0\label{eq:three-term_Z}
\end{equation}
on $\mathbb{Z}$. It simplifies to a recurrence equation with constant
coefficients,
\begin{equation}
q^{(\nu-1)/2}Q_{n+1}-(1+q^{\nu})Q_{n}+q^{(\nu+1)/2}Q_{n-1}=0.\label{eq:const_coeff}
\end{equation}
One can distinguish two independent solutions, $\{Q_{n}^{(1)}\}$
and $\{Q_{n}^{(2)}\}$, where
\begin{equation}
Q_{n}^{(1)}=\frac{q^{(1-\nu)n/2}-q^{\nu+(1+\nu)n/2}}{1-q^{\nu}}\,,\text{ }Q_{n}^{(2)}=q^{(1+\nu)n/2},\ n\in\mathbb{Z}.\label{eq:sols_Q1_Q2}
\end{equation}
Notice that $\{Q_{n}^{(1)}\}$ satisfies the initial conditions $Q_{-1}^{(1)}=0$,
$Q_{0}^{(1)}=1$, and so $Q_{n}^{(1)}=\hat{P}_{n}(0)$, $\forall n\geq0$.
On the other hand, $\{Q_{n}^{(2)}\}$ is always square summable over
$\mathbb{Z}_{+}$, and this is the so-called minimal solution at $+\infty$
since
\[
\lim_{n\to+\infty}\frac{Q_{n}^{(2)}}{Q_{n}}=0
\]
for every solution $\{Q_{n}\}$ of (\ref{eq:three-term_Z}) which
is linearly independent of $\{Q_{n}^{(2)}\}$. The Wronskian of $Q^{(1)}$
and $Q^{(2)}$ equals
\[
W_{n}(Q^{(1)},Q^{(2)})=1,\ \forall n\in\mathbb{Z},
\]
where $W_{n}(f,g):=\alpha_{n}(f_{n}g_{n+1}-g_{n}f_{n+1})$. Theorem~4
in \cite{BrownChristiansen} tells us that
\begin{equation}
\Dom T^{\text{F}}=\{f\in\ell^{2};\,\mathcal{T}f\in\ell^{2}\ \text{and}\ W_{\infty}(f,Q^{(2)})=0\}\label{eq:Dom_TF}
\end{equation}
where we put
\[
W_{\infty}(f,g)=\lim_{n\to\infty}W_{n}(f,g)
\]
for $f,g\in\mathbb{C}^{\infty}$ provided the limit exists. It is
useful to note, however, that discrete Green's formula implies existence
of the limit whenever $f,g\in\Dom T_{\text{max}}$, and then
\[
\langle T_{\text{max}}f,g\rangle-\langle f,T_{\text{max}}g\rangle=-W_{\infty}(\overline{f},g).
\]

We wish to determine all self-adjoint extensions of the closed positive
operator $T_{\text{min}}$. This is a standard general fact that the
deficiency indices of $T_{\text{min}}$ for any real symmetric Jacobi
matrix $\mathcal{T}$ of the form (\ref{eq:JacobiT}), with all $\alpha_{n}$'s
nonzero, are either $(0,0)$ or $(1,1)$. The latter case happens
if and only if for some $x\in\mathbb{C}$ all solutions of the second-order
difference equation
\begin{equation}
\alpha_{n}Q_{n+1}+(\beta_{n}-x)Q_{n}+\alpha_{n-1}Q_{n-1}=0\label{eq:three-term_Q_specpar_x}
\end{equation}
are square summable on $\mathbb{Z}_{+}$, and in that case this is
true for any value of the spectral parameter $x$ (see, for instance,
a detailed discussion in Section~2.6 of \cite{Teschl}).

Let us remark that a convenient description of the one-parameter family
of all self-adjoint extensions is also available if the deficiency
indices are $(1,1)$. Fix $x\in\mathbb{R}$ and any couple $Q^{(1)}$,
$Q^{(2)}$ of independent solutions of (\ref{eq:three-term_Q_specpar_x}).
Then all self-adjoint extensions of $T_{\text{min}}$ are operators
$\tilde{T}(\kappa)$ defined on the domains
\begin{equation}
\Dom\tilde{T}(\kappa)=\{f\in\ell^{2};\,\mathcal{T}f\in\ell^{2}\ \text{and}\ W_{\infty}(f,Q^{(1)})=\kappa\, W_{\infty}(f,Q^{(2)})\},\label{eq:Dom_Ttkappa}
\end{equation}
with $\kappa\in\mathbb{R}\cup\{\infty\}$. Moreover, all of them are
mutually different. Of course, $\tilde{T}(\kappa)f=\mathcal{T}f$,
$\forall f\in\Dom\tilde{T}(\kappa)$.

In our case we know, for $x=0$, a couple of solutions of (\ref{eq:three-term_Q_specpar_x})
explicitly, cf. (\ref{eq:sols_Q1_Q2}). From their form it becomes
obvious that $T_{\text{min}}=T_{\text{max}}$ is self-adjoint if and
only if $\nu\geq1$. With this choice of $Q^{(1)}$, $Q^{(2)}$ and
sticking to notation (\ref{eq:Dom_Ttkappa}), it is seen from (\ref{eq:Dom_TF})
that the Friedrichs extension $T^{\text{F}}$ coincides with $\tilde{T}(\infty)$.

\begin{lemma} \label{prop:DTmax_asympt} Suppose $0<\nu<1$. Then
every sequence $f\in\Dom T_{\text{max}}$ has the asymptotic expansion
\begin{equation}
f_{n}=C_{1}q^{(1-\nu)n/2}+C_{2}q^{(1+\nu)n/2}+o(q^{n})\,\ \text{as}\ n\to\infty,\label{eq:f_DTmax_asympt}
\end{equation}
where $C_{1},C_{2}\in\mathbb{C}$ are some constants. \end{lemma}

\begin{proof} Let $f\in\Dom T_{\text{max}}$. That means $f\in\ell^{2}$
and $\mathcal{A^{\text{H}}\mathcal{A}}f=h\in\ell^{2}$ where $\mathcal{A}$
is defined in (\ref{eq:Acal}), with $p=q^{-1/2}$, $w=q^{(1-\nu)/2}$
(then $\mathcal{T}=q^{\nu}\mathcal{A}^{\text{H}}\mathcal{A}$). Denote
$g=\mathcal{A}f$. Hence $\mathcal{A}^{\text{H}}g=h$ and, as already
observed in the course of the proof of Lemma~\ref{prop:Amin_Amax},
there exists a constant $\gamma$ such that
\[
g_{n}=\gamma q^{\nu n/2}+q^{n/2}\sum_{k=0}^{\infty}q^{(1-\nu)k/2}h_{n+k},\ \forall n.
\]
Furthermore, the relation $\mathcal{A}f=g$ can be inverted, 
\[
f_{n}=q^{(1-\nu)n/2}\sum_{k=0}^{n}q^{\nu k/2}g_{k},\ \forall n.
\]
Whence
\begin{eqnarray*}
f_{n} & = & \frac{\gamma}{1-q^{\nu}}\left(q^{(1-\nu)n/2}-q^{\nu+(1+\nu)n/2}\right)+q^{(1-\nu)n/2}\,\sum_{k=0}^{n}q^{(1+\nu)k/2}\sum_{j=0}^{\infty}q^{(1-\nu)j/2}h_{k+j}\\
 & = & C_{1}\, q^{(1-\nu)n/2}+C_{2}\, q^{(1+\nu)n/2}+q^{n}\,\zeta_{n}
\end{eqnarray*}
where
\[
C_{1}=\frac{\gamma}{1-q^{\nu}}+\sum_{k=0}^{\infty}q^{(1+\nu)k/2}\sum_{j=0}^{\infty}q^{(1-\nu)j/2}h_{k+j},\ C_{2}=-\frac{\gamma q^{\nu}}{1-q^{\nu}}\,,
\]
and
\[
\zeta_{n}=-\sum_{k=1}^{\infty}q^{(1+\nu)k/2}\sum_{j=0}^{\infty}q^{(1-\nu)j/2}h_{n+k+j}.
\]
Bearing in mind that $h\in\ell^{2}$ one concludes, with the aid of
the Schwarz inequality, that $\zeta_{n}\to0$ as $n\to\infty$. \end{proof}

With the knowledge of the asymptotic expansion established in Lemma~\ref{prop:DTmax_asympt}
one can formulate a somewhat simpler and more explicit description
of self-adjoint extensions of $T_{\text{min}}$.

\begin{proposition} \label{prop:sa_extns_Tmin} The operator $T_{\text{min}}\equiv T_{\text{min}}^{(\nu)}$,
with $\nu>0$, is self-adjoint if and only if $\nu\geq1$. If $0<\nu<1$
then all mutually different self-adjoint extensions of $T_{\text{min}}$
are parametrized by $\kappa\in P^{1}(\mathbb{R})\equiv\mathbb{R}\cup\{\infty\}$
as follows. For $f\in\Dom T_{\text{max}}$ let $C_{1}(f)$, $C_{2}(f)$
be the constants from the asymptotic expansion (\ref{eq:f_DTmax_asympt}),
i.e.
\[
C_{1}(f)=\lim_{n\to\infty}f_{n}q^{-(1-\nu)n/2}\,,\ C_{2}(f)=\lim_{n\to\infty}\big(f_{n}-C_{1}(f)q^{(1-\nu)n/2}\big)q^{-(1+\nu)n/2}\,.
\]
For $\kappa\in P^{1}(\mathbb{R})$, a self-adjoint extension $T(\kappa)$
of $T_{\text{min}}$ is a restriction of $T_{\text{max}}$ to the
domain
\begin{equation}
\Dom T(\kappa)=\{f\in\ell^{2};\,\mathcal{T}f\in\ell^{2}\ \text{and}\ C_{2}(f)=\kappa C_{1}(f)\}.\label{eq:Dom_Tkappa}
\end{equation}
In particular, $T(\infty)$ equals the Friedrichs extension $T^{\text{F}}$.
\end{proposition}

\begin{proof} Let $0<\nu<1$, $\{\zeta_{n}\}$ be a sequence converging
to zero (bounded would be sufficient) and $g^{(1)},g^{(2)},h\in\mathbb{C}^{\infty}$
be the sequences defined by
\[
g_{n}^{(1)}=q^{(1-\nu)n/2},\ g_{n}^{(2)}=q^{(1+\nu)n/2},\ h_{n}=q^{n}\zeta_{n},\ \forall n.
\]
Hence, referring to (\ref{eq:sols_Q1_Q2}),
\[
Q^{(1)}=\frac{1}{1-q^{\nu}}\,(g^{(1)}-q^{\nu}g^{(2)}),\ Q^{(2)}=g^{(2)}.
\]
One finds at once that $W_{\infty}(g^{(1)},h)=W_{\infty}(g^{(2)},h)=0$
and 
\[
W_{n}(g^{(1)},g^{(2)})=1-q^{\nu},\ \forall n.
\]
After a simple computation one deduces from (\ref{eq:Dom_Ttkappa})
that $f\in\Dom T_{\text{max}}$ belongs to $\Dom\tilde{T}(\tilde{\kappa})$
for some $\tilde{\kappa}\in P^{1}(\mathbb{R})$, i.e. $W_{\infty}(f,Q^{(1)})=\tilde{\kappa}W_{\infty}(f,Q^{(2)})$,
if and only if $C_{2}(f)=\kappa C_{1}(f)$, with $\kappa=-q^{\nu}-(1-q^{\nu})\tilde{\kappa}$.
In other words, $\tilde{T}(\tilde{\kappa})=T(\kappa)$. Since the
mapping
\[
P^{1}(\mathbb{R})\to P^{1}(\mathbb{R}):\tilde{\kappa}\mapsto\kappa=-q^{\nu}-(1-q^{\nu})\tilde{\kappa}
\]
is one-to-one, $P^{1}(\mathbb{R})\ni\kappa\mapsto T(\kappa)$ is another
parametrization of self-adjoint extensions of $T_{\text{min}}$. Particularly,
$\tilde{\kappa}=\infty$ maps to $\kappa=\infty$ and so $T(\infty)=T^{\text{F}}$.
\end{proof}

\begin{remark} One can also describe $\Dom T_{\text{min}}$. For
$\nu\geq1$ we simply have $T_{\min}=T_{\max}=T^{\text{F}}$. In the
case when $0<\nu<1$ it has been observed in \cite{BrownChristiansen}
that a sequence $f\in\Dom T_{\text{max}}$ belongs to $\Dom T_{\min}$
if and only if $W_{\infty}(f,g)=0$ for all $g\in\Dom T_{\text{max}}$.
But this is equivalent to the requirement $C_{1}(f)=C_{2}(f)=0$.
Thus one has
\begin{equation}
\Dom T_{\text{min}}=\{f\in\ell^{2};\,\mathcal{T}f\in\ell^{2}\ \text{and}\ \lim_{n\to\infty}f_{n}q^{-(1+\nu)n/2}=0\}.\label{eq:Dom_Tmin}
\end{equation}
\end{remark}

\subsection{The Green function and spectral properties}

For $\nu\geq1$ we shall write shortly $T\equiv T^{(\nu)}$ instead
of $T_{\min}=T_{\max}=T^{\text{F}}$. Referring to solutions (\ref{eq:sols_Q1_Q2})
we claim that the Green function (matrix) of $T$, if $\nu\geq1$,
or $T^{\text{F}}$, if $0<\nu<1$, reads
\begin{equation}
G_{j,k}=\begin{cases}
Q_{j}^{(1)}Q_{k}^{(2)} & \text{for}\text{ }j\leq k,\\
\noalign{\smallskip}Q_{k}^{(1)}Q_{j}^{(2)} & \text{for}\text{ }j>k.
\end{cases}\label{eq:Green_matrix}
\end{equation}

\begin{proposition} \label{prop:G_eq_Tinv} The matrix $(G_{j,k})$
defined in (\ref{eq:Green_matrix}) represents a Hilbert-Schmidt operator
$G\equiv G^{(\nu)}$ on $\ell^{2}$ with the Hilbert-Schmidt norm
\begin{equation}
\left\Vert G\right\Vert _{\text{HS}}^{\,\,2}=\frac{1+q^{2+\nu}}{(1-q^{2})(1-q^{1+\nu})^{2}(1-q^{2+\nu})}\,.\label{eq:G_HSnorm}
\end{equation}
The operator $G$ is positive and one has, $\forall f\in\ell^{2}$,
\[
\langle f,Gf\rangle=\sum_{k=0}^{\infty}q^{k}\bigg|\sum_{j=0}^{\infty}q^{(1+\nu)j/2}f_{k+j}\bigg|^{2}.
\]
Moreover, the inverse $G^{-1}$ exists and equals $T$, if $\nu\geq1$,
or $T^{\text{F}}$, if $0<\nu<1$. \end{proposition}

\begin{proof} As is well known, if $T_{\text{min}}$ is not self-adjoint
then the resolvent of any of its self-adjoint extensions is a Hilbert-Schmidt
operator \cite[Lemma~2.19]{Teschl}. But in our case the resolvent
is claimed to be Hilbert-Schmidt for $\nu\geq1$ as well. One can
directly compute the Hilbert-Schmidt norm of $G$ for any $\nu>0$,
\begin{eqnarray*}
\sum_{j=0}^{\infty}\sum_{k=0}^{\infty}G_{j,k}^{\,\,2} & = & \sum_{j=0}^{\infty}(Q_{j}^{(1)})^{2}(Q_{j}^{(2)})^{2}+2\sum_{j=0}^{\infty}\sum_{k=j+1}^{\infty}(Q_{j}^{(1)})^{2}(Q_{k}^{(2)})^{2}\\
\noalign{\smallskip} & = & \frac{1+q^{1+\nu}}{1-q^{1+\nu}}\,\sum_{j=0}^{\infty}\left(\frac{1-q^{\nu\,(j+1)}}{1-q^{\nu}}\right)^{\!2}q^{2j}.
\end{eqnarray*}
Thus one obtains (\ref{eq:G_HSnorm}). Hence the Green matrix unambiguously
defines a self-adjoint compact operator $G$ on $\ell^{2}$.

Concerning the formula for the quadratic form one has to verify that,
for all $m,n\in\mathbb{Z}_{+}$, $m\leq n$,
\[
Q_{m}^{(1)}Q_{n}^{(2)}=\sum_{k=0}^{\infty}q^{k}\!\left(\sum_{j=0}^{\infty}q^{(1+\nu)j/2}\delta_{m,k+j}\right)\!\left(\sum_{j=0}^{\infty}q^{(1+\nu)j/2}\delta_{n,k+j}\right)\!.
\]
But this can be carried out in a straightforward manner.

A simple computation shows that for any $f\in\mathbb{C}^{\infty}$
and $n,N\in\mathbb{Z}_{+}$, $n<N$,
\[
Q_{n}^{(2)}\sum_{k=0}^{n}Q_{k}^{(1)}(\mathcal{T}f)_{k}+Q_{n}^{(1)}\sum_{k=n+1}^{N}Q_{k}^{(2)}(\mathcal{T}f)_{k}=f_{n}-Q_{n}^{(1)}\alpha_{N}\big(Q_{N+1}^{(2)}f_{N}-Q_{N}^{(2)}f_{N+1}\big).
\]
Considering the limit $N\to\infty$ one finds that, for a given $f\in\Dom T_{\max}$,
the equality $G\mathcal{T}f=f$ holds iff $W_{\infty}(f,Q^{(2)})=0$.
According to (\ref{eq:Dom_TF}), this condition determines the domain
of the Friedrichs extension $T^{\text{F}}$. Hence $GT^{\text{F}}\subset I$
(the identity operator).

Furthermore, one readily verifies that, for all $f\in\ell^{2}$, $\mathcal{T}Gf=f$.
We still have to check that $\Ran G\subset\Dom T^{\text{F}}$. But
using the equality $W_{\infty}(Q^{(1)},Q^{(2)})=1$ one computes,
for $f\in\ell^{2}$ and $n\in\mathbb{Z}_{+}$,
\[
W_{n}(Gf,Q^{(2)})=\sum_{k=n+1}^{\infty}Q_{k}^{(2)}f_{k}\to0\,\ \text{as}\ n\to\infty,
\]
since $Q^{(2)}\in\ell^{2}$. Hence $T^{\text{F}}G=I$. We conclude
that $G^{-1}=T^{\text{F}}$ (remember that we have agreed to write
$T^{\text{F}}=T$ for $\nu\geq1$). \end{proof}

Considering the case $\nu\geq1$, the fact that the Jacobi operator
$T$ is positive and $T^{-1}$ is compact has some well known consequences
for its spectral properties. The same conclusions can be made for
$0<\nu<1$ provided we replace $T$ by $T^{\text{F}}$. And from the
general theory of self-adjoint extensions one learns that $T(\kappa)$,
for $\kappa\in\mathbb{R}$, has similar properties as $T^{\text{F}}$
\cite[Theorem~8.18]{Weidmann}.

\begin{proposition} \label{prop:T_pp_semibound} The spectrum of
any of the operators $T$, if $\nu\geq1$, or $T(\kappa)$, with arbitrary
$\kappa\in P^{1}(\mathbb{R})$, if $0<\nu<1$, is pure point and bounded
from below, with all eigenvalues being simple and without finite accumulation
points. Moreover, the operator $T$, for $\nu\geq1$, or $T^{\text{F}}$,
for $0<\nu<1$, is positive definite and one has the following lower
bound on the spectrum, i.e. on the smallest eigenvalue $\xi_{1}\equiv\xi_{1}^{(\nu)}$,
\[
\xi_{1}^{\,2}\geq\frac{(1-q^{2})(1-q^{1+\nu})^{2}(1-q^{2+\nu})}{1+q^{2+\nu}}\,.
\]
\end{proposition}

\begin{proof} This is a simple general fact that all formal eigenvectors
of the Jacobi matrix $\mathcal{T}$ are unique up to a multiplier
\cite{Akhiezer}. By Proposition~\ref{prop:G_eq_Tinv}, $(T^{\text{F}})^{-1}$
is compact and therefore the spectrum of $T^{\text{F}}$ is pure point
and with eigenvalues accumulating only at infinity.

For $0<\nu<1$, the deficiency indices of $T_{\text{\text{min}}}$
are $(1,1)$. Whence, by the general spectral theory, if $T^{\text{F}}$
has an empty essential spectrum then the same is true for all other
self-adjoint extensions $T(\kappa)$, $\kappa\in\mathbb{R}$. Moreover,
there is at most one eigenvalue of $T(\kappa)$ below $\xi_{1}:=\min\spec(T^{\text{F}})$,
see \cite[\S~8.3]{Weidmann}. Referring once more to Proposition~\ref{prop:G_eq_Tinv}
one has
\[
\min\spec(T^{\text{F}})=(\max\spec(G))^{-1}\geq\|G\|_{\text{HS}}^{\,-1}.
\]
In view of (\ref{eq:G_HSnorm}), one obtains the desired estimate
on $\xi_{1}$. \end{proof}

\subsection{More details on the indeterminate case}

In this subsection we confine ourselves to the case $0<\nu<1$ and
focus on some general spectral properties of the self-adjoint extensions
$T(\kappa)$, $\kappa\in P^{1}(\mathbb{R})$, in addition to those
already mentioned in Proposition~\ref{prop:T_pp_semibound}. The
spectra of any two different self-adjoint extensions of $T_{\text{min}}$
are known to be disjoint (see, for instance, proof of Theorem~4.2.4
in \cite{Akhiezer}). Moreover, the eigenvalues of such a couple of
self-adjoint extensions interlace (see \cite{SilvaWeder} and references
therein, or this can also be deduced from general properties of self-adjoint
extensions with deficiency indices $(1,1)$ \cite[\S~8.3]{Weidmann}).
It is useful to note, too, that every $x\in\mathbb{R}$ is an eigenvalue
of a unique self-adjoint extension $T(\kappa)$, $\kappa\in P^{1}(\mathbb{R})$
\cite[Theorem~4.11]{Simon}.

For positive symmetric operators there exists another powerful theory
of self-adjoint extensions due to Birman, Krein and Vishik based on
the analysis of associated quadratic forms. A clear exposition of
the theory can be found in \cite{AlonsoSimon}. Its application to
our case, with deficiency indices $(1,1)$, is as follows. A crucial
role is played by the null space of $T_{\text{max}}=T_{\text{min}}^{\,\ast}$
which we denote by
\[
\sN:=\Ker T_{\text{max}}=\mathbb{C}Q^{(1)}
\]
(recall that $Q_{n}^{(1)}=\hat{P}_{n}(0)$, $\forall n\in\mathbb{Z}_{+}$).
Let $\mathfrak{t}_{\infty}=\mathfrak{t}$ be the quadratic form associated
with the Friedrichs extension $T^{\text{F}}$. Remember that the domain
of $\mathfrak{t}$ has been specified in (\ref{eq:Dom_t}). All other
self-adjoint extensions of $T_{\text{min}}$, except of $T^{\text{F}}$,
are in one-to-one correspondence with real numbers $\tau$. The corresponding
associated quadratic forms $\mathfrak{t}_{\tau}$, $\tau\in\mathbb{R}$,
have all the same domain,
\begin{equation}
\Dom\mathfrak{t}_{\tau}=\Dom\mathfrak{t}_{\infty}\,\dot{+}\,\sN\label{eq:Dom_t_tau}
\end{equation}
(a direct sum), and for $f\in\Dom\mathfrak{t}_{\infty}$, $\lambda\in\mathbb{C}$,
one has
\begin{equation}
\mathfrak{t}_{\tau}(f+\lambda Q^{(1)})=\mathfrak{t}_{\infty}(f)+\tau|\lambda|^{2}.\label{eq:t_tau}
\end{equation}

Our next task is to relate the self-adjoint extensions $T(\kappa)$
described in Proposition~\ref{prop:sa_extns_Tmin} to the quadratic
forms $\mathfrak{t}_{\tau}$.

\begin{proposition} \label{prop:form_Tkappa} The quadratic form
associated with a self-adjoint extension $T(\kappa)$, $\kappa\in\mathbb{R}$,
is $\mathfrak{t}_{\tau}$ defined in (\ref{eq:Dom_t_tau}), (\ref{eq:t_tau}),
with $\tau=(\kappa+q^{\nu})/(1-q^{\nu})$. \end{proposition}

\begin{proof} Let $\kappa\in\mathbb{R}$ and $\sigma$ be the real
parameter such that $\mathfrak{t}_{\sigma}$ is the quadratic form
associated with $T(\kappa)$. Recall (\ref{eq:sols_Q1_Q2}). One has
$\mathcal{T}Q^{(1)}=0$ and $(\mathcal{T}Q^{(2)})_{n}=\delta_{n,0}$,
$\forall n\in\mathbb{Z}_{+}$. According to (\ref{eq:Dom_TF}),
\[
Q^{(2)}\in\Dom T(\infty)\subset\Dom\mathfrak{t}_{\infty}.
\]
One computes $\mathfrak{t}_{\infty}(Q^{(2)})=\langle Q^{(2)},\mathcal{T}Q^{(2)}\rangle=1$.
Let $\tau=(\kappa+q^{\nu})/(1-q^{\nu})$ and
\[
h=\tau Q^{(2)}+Q^{(1)}\in\Dom\mathfrak{t}_{\infty}+\mathbb{C}Q^{(1)}=\Dom\mathfrak{t}_{\sigma}.
\]
Then $(1-q^{\nu})h_{n}=q^{(1-\nu)n/2}+\kappa q^{(1+\nu)n/2}$, $\forall n\in\mathbb{Z}_{+}$.
Hence, in virtue of (\ref{eq:Dom_Tkappa}), $h\in\Dom T(\kappa)$,
and, referring to (\ref{eq:t_tau}),
\[
\tau^{2}+\sigma=\mathfrak{t}_{\sigma}(h)=\langle h,T(\kappa)h\rangle=\langle h,\mathcal{T}h\rangle=\tau(\tau+1).
\]
Whence $\sigma=\tau$. \end{proof}

Now we are ready to describe the announced additional spectral properties
of $T(\kappa)$. The terminology and basic results concerning quadratic
(sesquilinear) forms used below are taken from Kato \cite{Kato}.

\begin{lemma} \label{prop:family_type_a} Let $\sS$ and $\sB$ be
linear subspaces in a Hilbert space $\sH$ such that $\sS\cap\sB=\{0\}$,
and let $\mathfrak{s}$ and $\mathfrak{b}$ be positive quadratic
forms on $\sS$ and $\sB$, respectively. Denote by $\tilde{\mathfrak{s}}$
and $\tilde{\mathfrak{b}}$ the extensions of these forms to $\sS+\sB$
defined by
\[
\forall\varphi\in\sS,\,\forall\eta\in\sB,\ \tilde{\mathfrak{s}}(\varphi+\eta)=\mathfrak{s}(\varphi)\ \text{and}\,\ \tilde{\mathfrak{b}}(\varphi+\eta)=\mathfrak{b}(\eta),
\]
and assume that, for every $\rho\in\mathbb{R}$, the form $\tilde{\mathfrak{s}}+\rho\tilde{\mathfrak{b}}$
is semibounded and closed. Then, for any $\tau\in\mathbb{C}$, the
form $\tilde{\mathfrak{s}}+\tau\tilde{\mathfrak{b}}$ is sectorial
and closed. In particular, if $\sS+\sB$ is dense in $\sH$ then $\tilde{\mathfrak{s}}+\tau\tilde{\mathfrak{b}}$,
$\tau\in\mathbb{C}$, is a holomorphic family of forms of type (a)
in the sense of Kato. \end{lemma}

\begin{proof} Fix $\tau\in\mathbb{C}$, $\theta\in(\pi/4,\pi/2)$,
and choose $\gamma_{1},\gamma_{2},\gamma_{3}\in\mathbb{R}$ so that
\[
\tilde{\mathfrak{s}}+\Re(\tau)\tilde{\mathfrak{b}}\geq\gamma_{1},\,\ (\tan(\theta)-1)\tilde{\mathfrak{s}}+\tan(\theta)\Re(\tau)\tilde{\mathfrak{b}}\geq\gamma_{2},\,\ \tilde{\mathfrak{s}}-|\Im(\tau)|\tilde{\mathfrak{b}}\geq\gamma_{3}.
\]
Let $\gamma=\min\{\gamma_{1},\cot(\theta)(\gamma_{2}+\gamma_{3})\}$.
Then $\Re(\tilde{\mathfrak{s}}+\tau\tilde{\mathfrak{b}})\geq\gamma$
and, for any couple $\varphi\in\sS$, $\eta\in\sB$,
\begin{eqnarray*}
|\Im(\tilde{\mathfrak{s}}+\tau\tilde{\mathfrak{b}})(\varphi+\eta)|=|\Im\tau|\mathfrak{b}(\eta) & \!\leq & \!\mathfrak{s}(\varphi)-\gamma_{3}\|\varphi+\eta\|^{2}\\
 & \!\leq & \!\tan(\theta)\big(\mathfrak{s}(\varphi)+\Re(\tau)\mathfrak{b}(\eta)\big)-(\gamma_{2}+\gamma_{3})\|\varphi+\eta\|^{2}\\
 & \!\leq & \!\tan(\theta)\big(\mathfrak{s}(\varphi)+\Re(\tau)\mathfrak{b}(\eta)-\gamma\|\varphi+\eta\|^{2}\big).
\end{eqnarray*}
This estimates show that $\tilde{\mathfrak{s}}+\tau\tilde{\mathfrak{b}}$
is sectorial. Finally, a sectorial form is known to be closed if and
only if its real part is closed. \end{proof}

\begin{proposition} \label{prop:xi_n} Let $\{\xi_{n}(\kappa);\, n\in\mathbb{N}\}$
be the eigenvalues of $T(\kappa)$, $\kappa\in P^{1}(\mathbb{R})$,
ordered increasingly. Then for every $n\in\mathbb{N}$, $\xi_{n}(\kappa)$
is a real-analytic strictly increasing function on $\mathbb{R}$,
and one has, $\forall\kappa\in\mathbb{R}$,
\begin{equation}
\xi_{1}(\kappa)<\xi_{1}(\infty)<\xi_{2}(\kappa)<\xi_{2}(\infty)<\xi_{3}(\kappa)<\xi_{3}(\infty)<\ldots.\label{eq:xi_ineqs}
\end{equation}
Moreover,
\begin{equation}
\lim_{\kappa\to-\infty}\xi_{1}(\kappa)=-\infty,\ \lim_{\kappa\to-\infty}\xi_{n}(\kappa)=\xi_{n-1}(\infty)\ \text{for}\ n\geq2,\ \lim_{\kappa\to+\infty}\xi_{n}(\kappa)=\xi_{n}(\infty)\ \text{for}\ n\geq1.\label{eq:lim_xi}
\end{equation}
\end{proposition}

\begin{proof} The Friedrichs extension of a positive operator is
maximal in the form sense among all self-adjoint extensions of that
operator \cite{AlonsoSimon}. Particularly,
\[
\xi_{1}(\kappa)=\min(\spec T(\kappa))\leq\xi_{1}(\infty)=\min(\spec T(\infty)).
\]
But as already remarked above, the eigenvalues of $T(\kappa)$ and
$T(\infty)$ interlace and so we have (\ref{eq:xi_ineqs}).

Referring to (\ref{eq:Dom_t_tau}), (\ref{eq:t_tau}), the property
$\kappa_{1},\kappa_{2}\in\mathbb{R}$, $\kappa_{1}<\kappa_{2}$ clearly
implies $\mathfrak{t}_{\tau(\kappa_{1})}<\mathfrak{t}_{\tau(\kappa_{2})}$
where $\tau(\kappa)=(\kappa+q^{\nu})/(1-q^{\nu})$. In virtue of Proposition~\ref{prop:form_Tkappa}
and the min-max principle, $\xi_{n}(\kappa_{1})\leq\xi_{n}(\kappa_{2})$,
$\forall n\in\mathbb{N}$. But the spectra of $T(\kappa_{1})$ and
$T(\kappa_{2})$ are disjoint and so the functions $\xi_{n}(\kappa)$
are strictly increasing on $\mathbb{R}$.

One can admit complex values for the parameter $\tau$ in (\ref{eq:Dom_t_tau}),
(\ref{eq:t_tau}). Then, according to Lemma~\ref{prop:family_type_a},
the family of forms $\mathfrak{t}_{\tau}$, $\tau\in\mathbb{C}$,
is of type (a) in the sense of Kato. Referring once more to Proposition~\ref{prop:form_Tkappa}
one infers from \cite[Theorem~VII-4.2]{Kato} that the family of self-adjoint
operators $T(\kappa)$, $\kappa\in\mathbb{R}$, extends to a holomorphic
family of operators on $\mathbb{C}$. This implies that for any bounded
interval $K\subset\mathbb{R}$ there exists an open neighborhood $D$
of $K$ in $\mathbb{C}$ and $\mathbb{\rho\in\mathbb{R}}$ sufficiently
large so that the resolvents $(T(\kappa)+\rho)^{-1}$, $\kappa\in K$,
extend to a holomorphic family of bounded operators on $D$. In addition
we know that, for every fixed $n\in\mathbb{N}$ and $\kappa\in\mathbb{R}$,
the $n$th eigenvalue of $T(\kappa)$ is simple and isolated. By the
analytic perturbation theory \cite[\S~VII.3]{Kato}, $\xi_{n}(\kappa)$
is an analytic function on $\mathbb{R}$.

Finally we note that every $x\in\mathbb{R}$ is an eigenvalue of $T(\kappa)$
for some (in fact, unambiguous) $\kappa\in\mathbb{R}$ and so the
range $\xi_{n}(\mathbb{R})$ must exhaust the entire interval either
$(-\infty,\xi_{1}(\infty))$, if $n=1$, or $(\xi_{n-1}(\infty),\xi_{n}(\infty))$,
if $n>1$. This clearly means that (\ref{eq:lim_xi}) must hold. \end{proof}

\begin{remark} As noted in \cite[Theorem~2.15]{AlonsoSimon}, $\xi_{1}(\kappa)$
is a concave function. \end{remark}

\section{The characteristic function}

\subsection{A construction of the characteristic function for $\nu>0$}

Recall (\ref{eq:Pn_def}), (\ref{eq:Phat_def}). Observe that the
sequence $\{\hat{P}_{n}(x)\}$ obeys the relation
\begin{equation}
\hat{P}_{n}(x)=Q_{n}^{(1)}-xq^{(1-\nu)/2}\,\sum_{k=0}^{n}Q_{n-k-1}^{(1)}q^{k}\hat{P}_{k}(x)\ \text{ }\text{for}\text{ }n\geq-1.\label{eq:2nd_recurr_Phat}
\end{equation}
This relation already implies that $\hat{P}_{-1}(x)=0$, $\hat{P}_{0}(x)=1$.
Notice also that the last term in the sum, with $k=n$, is zero and
so (\ref{eq:2nd_recurr_Phat}) is in fact a recurrence for $\{\hat{P}_{n}(x)\}$.
Equation (\ref{eq:2nd_recurr_Phat}) is pretty standard. Nevertheless,
one may readily verify it by checking that this recurrence implies
the original defining recurrence, i.e. the formal eigenvalue equation
(\ref{eq:recurr_orig}) which can be rewritten as follows
\begin{equation}
q^{(\nu-1)/2}\hat{P}_{n+1}(x)-(1+q^{\nu})\hat{P}_{n}(x)+q^{(\nu+1)/2}\hat{P}_{n-1}(x)=-xq^{n}\hat{P}_{n}(x),\text{ }\forall n\geq0.\label{eq:eigeval_recurr_Phat}
\end{equation}
Actually, from (\ref{eq:2nd_recurr_Phat}) one derives that
\begin{eqnarray*}
q^{(\nu-1)/2}\hat{P}_{n+1}(x)-\hat{P}_{n}(x) & = & q^{(\nu-1)/2}Q_{n+1}^{(1)}-Q_{n}^{(1)}\\
 &  & -x\sum_{k=0}^{n}\big(Q_{n-k}^{(1)}-q^{(1-\nu)/2}Q_{n-k-1}^{(1)}\big)q^{k}\hat{P}_{k}(x)\\
 & = & q^{\nu+(1+\nu)n/2}-xq^{(1+\nu)n/2}\,\sum_{k=0}^{n}q^{(1-\nu)k/2}\hat{P}_{k}(x)
\end{eqnarray*}
and so
\[
q^{(\nu-1)/2}\hat{P}_{n+1}(x)-\hat{P}_{n}(x)-q^{(\nu+1)/2}\big(q^{(\nu-1)/2}\hat{P}_{n}(x)-\hat{P}_{n-1}(x)\big)=-xq^{n}\hat{P}_{n}(x),
\]
as claimed.

\begin{proposition} \label{prop:Phi} The sequence of polynomials
$\{q^{(\nu-1)n/2}\hat{P}_{n}(x);\, n\in\mathbb{Z}_{+}\}$ converges
locally uniformly on $\mathbb{C}$ to an entire function $\Phi(x)\equiv\Phi^{(\nu)}(x;q)$.
Moreover, $\Phi(x)$ fulfills
\begin{equation}
\Phi(x)=\frac{1}{1-q^{\nu}}\!\left(1-x\sum_{k=0}^{\infty}q^{(1+\nu)k/2}\hat{P}_{k}(x)\right)\label{eq:Phi}
\end{equation}
and one has, $\forall n\in\mathbb{Z}$, $n\geq-1$, 
\begin{equation}
\hat{P}_{n}(x)=(1-q^{\nu})\Phi(x)Q_{n}^{(1)}+xQ_{n}^{(2)}\sum_{k=0}^{n}Q_{k}^{(1)}\hat{P}_{k}(x)+xQ_{n}^{(1)}\sum_{k=n+1}^{\infty}Q_{k}^{(2)}\hat{P}_{k}(x).\label{eq:Phat_Phi}
\end{equation}
\end{proposition}

\begin{proof} Denote (temporarily) $H_{n}(x)=q^{(\nu-1)n/2}\hat{P}_{n}(x)$,
$n\in\mathbb{Z}_{+}$. Then (\ref{eq:2nd_recurr_Phat}) means that
\begin{equation}
(1-q^{\nu})H_{n}(x)=1-q^{\nu(n+1)}-x\sum_{k=0}^{n-1}(1-q^{\nu(n-k)})q^{k}H_{k}(x),\ n\in\mathbb{Z}_{+}.\label{eq:Hn_recurr}
\end{equation}
Proceeding by mathematical induction in $n$ one can show that, $\forall n\in\mathbb{Z}_{+}$,
\begin{equation}
|H_{n}(x)|\leq\frac{(-a;q)_{n}}{1-q^{\nu}}\,\ \text{where}\ a=\frac{|x|}{1-q^{\nu}}\,\ \text{and}\ (-a;q)_{n}=\prod_{k=0}^{n-1}(1+q^{k}a)\label{eq:Hn_estim}
\end{equation}
is the $q$-Pochhammer symbol. This is obvious for $n=0$. For the
the induction step it suffices to notice that (\ref{eq:Hn_recurr})
implies
\[
|H_{n}(x)|\leq\frac{1}{1-q^{\nu}}+a\sum_{k=0}^{n-1}q^{k}|H_{k}(x)|.
\]
Moreover,
\[
1+a\sum_{k=0}^{n-1}q^{k}(-a;q)_{k}=(-a;q)_{n}.
\]

From the estimate (\ref{eq:Hn_estim}) one infers that $\{H_{n}(x)\}$
is locally uniformly bounded on $\mathbb{C}$. Consequently, from
(\ref{eq:Hn_recurr}) it is seen that the RHS converges as $n\to\infty$
and so $H_{n}(x)\to\Phi(x)$ pointwise. This leads to identity (\ref{eq:Phi}).
Furthermore, one can rewrite (\ref{eq:2nd_recurr_Phat}) as follows
\begin{eqnarray*}
\hat{P}_{n}(x) & = & Q_{n}^{(1)}-x\,\sum_{k=0}^{n}\frac{q^{(1-\nu)n/2}q^{(1+\nu)k/2}-q^{(1+\nu)n/2}q^{(1-\nu)k/2}}{1-q^{\nu}}\,\hat{P}_{k}(x)\\
 & = & Q_{n}^{(1)}+xQ_{n}^{(2)}\,\sum_{k=0}^{n}Q_{k}^{(1)}\hat{P}_{k}(x)-xQ_{n}^{(1)}\,\sum_{k=0}^{n}Q_{k}^{(2)}\hat{P}_{k}(x)\\
 & = & \left(1-x\,\sum_{k=0}^{\infty}Q_{k}^{(2)}\hat{P}_{k}(x)\right)\! Q_{n}^{(1)}+xQ_{n}^{(2)}\sum_{k=0}^{n}Q_{k}^{(1)}\hat{P}_{k}(x)+xQ_{n}^{(1)}\sum_{k=n+1}^{\infty}Q_{k}^{(2)}\hat{P}_{k}(x)
\end{eqnarray*}
Taking into account (\ref{eq:Phi}) one arrives at (\ref{eq:Phat_Phi}).

Finally, from the locally uniform boundedness and Montel's theorem
it follows that the convergence of $\{H_{n}(x)\}$ is even locally
uniform and so $\Phi(x)$ is an entire function. \end{proof}

It turns out that $\Phi(x)$ may be called the characteristic function
of the Jacobi operator $T$, if $\nu\geq1$, or the Friedrichs extension
$T^{\text{F}}$, if $0<\nu<1$.

\begin{lemma} \label{prop:Gtilde_f} Assume that $\nu\geq1$. Suppose
further that $f\in\mathbb{C}^{\infty}$, $\{q^{-\sigma_{0}\, n}f_{n}\}$
is bounded for some $\sigma_{0}>-(\nu+1)/2$, and $f=x\tilde{G}f$
for some $x\in\mathbb{R}$ where
\begin{equation}
(\tilde{Gf})_{n}:=Q_{n}^{(2)}\,\sum_{k=0}^{n}Q_{k}^{(1)}f_{k}+Q_{n}^{(1)}\sum_{k=n+1}^{\infty}Q_{k}^{(2)}f_{k},\ n\in\mathbb{Z}_{+}.\label{eq:Gtilde_f}
\end{equation}
Then the sequence $\{q^{-\sigma n}f_{n}\}$ is bounded for every $\sigma<(\nu+1)/2$.
In particular, $f\in\ell^{2}$. \end{lemma}

\begin{proof} Put
\[
S=\{\sigma>-(\nu+1)/2;\,\{q^{-\sigma n}f_{n}\}\in\ell^{\infty}\},\ \sigma_{\ast}=\sup S.
\]
Notice that, by the assumptions, $S\neq\emptyset$ and the definition
of $\tilde{G}f$ makes good sense. We have to show that $\sigma_{\ast}\geq(1+\nu)/2$.
Let us assume the contrary.

We claim that if $\sigma\in S$ and $\sigma<(\nu-1)/2$ then $\sigma+1\in S$.
In particular, $\sigma_{\ast}\geq(\nu-1)/2$. In fact, write $f_{n}=q^{\sigma n}h_{n}$,
$h\in\ell^{\infty}$. From (\ref{eq:Gtilde_f}) one derives the estimate
\[
|(\tilde{G}f)_{n}|\leq\frac{\|h\|_{\infty}}{1-q^{\nu}}\!\left(q^{(\nu+1)n/2}\sum_{k=0}^{n-1}q^{(\sigma+(1-\nu)/2)k}+\frac{q^{(\sigma+1)n}}{1-q^{\sigma+(1+\nu)/2}}\right)\!.
\]
From here one deduces that there exists a constant $C\geq0$ such
that
\[
\forall n,\ |f_{n}|=|x(\tilde{G}f)_{n}|\leq Cq^{(\sigma+1)n},
\]
as claimed.

Choose $\sigma$ such that $\sigma_{\ast}<\sigma<(\nu+1)/2$. Then
\[
-\frac{\nu+1}{2}\leq\frac{\nu-1}{2}-1\leq\sigma_{\ast}-1<\sigma-1<\frac{\nu-1}{2}\leq\sigma_{\ast}
\]
and so $\sigma-1\in S$. But in that case $\sigma\in S$ as well,
a contradiction. \end{proof}

\begin{proposition} \label{prop:specT_Phi} If $\nu\geq1$, the spectrum
of $T$ coincides with the zero set of $\Phi(x)$. If $0<\nu<1$ then
$\spec T(\kappa)$, $\kappa\in P^{1}(\mathbb{R})$, consists of the
roots of the characteristic equation
\[
\kappa\Phi(x)+\Psi(x)=0
\]
where
\begin{equation}
\Psi(x)\equiv\Psi^{(\nu)}(x;q)=\frac{1}{1-q^{\nu}}\!\left(q^{\nu}-x\sum_{k=0}^{\infty}q^{(1-\nu)k/2}\hat{P}_{k}(x)\right)\!.\label{eq:Psi}
\end{equation}
In particular, the spectrum of $T^{\text{F}}=T(\infty)$ equals the
zero set of $\Phi(x)$. \end{proposition}

\begin{proof} From Proposition~\ref{prop:T_pp_semibound} we already
know that the spectrum of $T$ (or $T(\kappa)$) is pure point and
with no finite accumulation points. Assume first that $\nu\geq1$.
According to Proposition~\ref{prop:sa_extns_Tmin}, we are dealing
with the determinate case and so $x$ is an eigenvalue of $T$ if
and only if the formal eigenvector $\hat{P}(x)=\{\hat{P}_{n}(x)\}$
is square summable. If $\hat{P}(x)\in\ell^{2}$ then $q^{(\nu-1)n/2}\hat{P}_{n}(x)\to0$
as $n\to\infty$ and so $\Phi(x)=0$ (see Proposition~\ref{prop:Phi}).
Conversely, if $\Phi(x)=0$ then (\ref{eq:Phat_Phi}) tells us that
$\hat{P}=x\tilde{G}\hat{P}$, cf. (\ref{eq:Gtilde_f}). By Lemma~\ref{prop:Gtilde_f},
$\hat{P}(x)\in\ell^{2}$.

Assume now that $0<\nu<1$. This the indeterminate case meaning that
$\hat{P}(x)$ is square summable for all $x\in\mathbb{C}$. Hence
$x$ is an eigenvalue of $T(\kappa)$ iff $\hat{P}(x)\in\Dom T(\kappa)$.
Recall that $T(\kappa)$ is defined in Proposition~\ref{prop:sa_extns_Tmin}.
From (\ref{eq:Phat_Phi}) one derives the asymptotic expansion
\[
\hat{P}_{n}(x)=\Phi(x)(q^{(1-\nu)n/2}-q^{\nu+(1+\nu)n/2})+xq^{(1+\nu)n/2}\sum_{k=0}^{\infty}Q_{k}^{(1)}\hat{P}_{k}(x)+o(q^{n})\,\ \text{as}\ n\to\infty.
\]
From here it is seen that $\hat{P}(x)$ fulfills the boundary condition
in (\ref{eq:Dom_Tkappa}) if and only if $x$ solves the equation
\[
(\kappa+q^{\nu})\Phi(x)-x\langle Q^{(1)},\hat{P}(x)\rangle=0.
\]
Referring to (\ref{eq:sols_Q1_Q2}) one finds that $x\langle Q^{(1)},\hat{P}(x)\rangle=q^{\nu}\Phi(x)-\Psi(x)$.
\end{proof}

\begin{proposition} \label{prop:char_fceT} For $\nu>0$ one has
\[
\Phi(x)=\frac{1}{1-q^{\nu}}\,\,_{1}\phi_{1}(0;q^{\nu+1};q,x)=\frac{(q;q)_{\infty}}{(q^{\nu};q)_{\infty}}\, q^{\nu/2}x^{-\nu/2}J_{\nu}(q^{-1/2}\sqrt{x};q),
\]
and for $0<\nu<1$,
\begin{equation}
\Psi(x)=\frac{q^{\nu}}{1-q^{\nu}}\,\,_{1}\phi_{1}(0;q^{1-\nu};q,q^{-\nu}x)=-\frac{(q;q)_{\infty}}{(q^{-\nu};q)_{\infty}}\, q^{-\nu(\nu+1)/2}x^{\nu/2}J_{-\nu}(q^{-(\nu+1)/2}\sqrt{x};q).\label{eq:Psi-qBessel}
\end{equation}

If $\Phi(x)=0$ and so $x$ is an eigenvalue of $T$, provided $\nu>0$,
or $T^{\text{F}}$, provided $0<\nu<1$, then $x>0$ and the components
of a corresponding eigenvector can be chosen as
\begin{equation}
u_{k}(x)=q^{k/2}\, J_{\nu}(q^{k/2}\sqrt{x};q)=C\, q^{(1+\nu)k/2}\,_{1}\phi_{1}(0;q^{\nu+1};q,q^{k+1}x),\ k\in\mathbb{Z}_{+},\label{eq:egvector_TF}
\end{equation}
where $C=x^{\nu/2}\,(q^{1+\nu};q)_{\infty}/(q;q)_{\infty}$.\smallskip{}

If $0<\nu<1$, $\kappa\in\mathbb{R}$ and $\kappa\Phi(x)+\Psi(x)=0$
and so $x$ is an eigenvalue of $T(\kappa)$ then the components of
a corresponding eigenvector can be chosen as
\begin{eqnarray*}
u_{k}(\kappa,x) & = & q^{k/2}\!\left(\kappa J_{\nu}(q^{k/2}\sqrt{x};q)-\frac{(q^{\nu};q)_{\infty}}{(q^{-\nu};q)_{\infty}}\, q^{-\nu(\nu+2)/2}x^{\nu}J_{-\nu}(q^{(k-\nu)/2}\sqrt{x};q)\right)\\
\noalign{\smallskip} & = & C\left(\kappa q^{(1+\nu)k/2}\,_{1}\phi_{1}(0;q^{\nu+1};q,q^{k+1}x)+q^{(1-\nu)k/2}\,_{1}\phi_{1}(0;q^{1-\nu};q,q^{k+1-\nu}x)\right)\!,
\end{eqnarray*}
with $k\in\mathbb{Z}_{+}$($C$ is the same as above). \end{proposition}

\begin{lemma} For every $m\in\mathbb{Z}_{+}$ and $\sigma>0$,

\begin{equation}
\sum_{k=0}^{\infty}q^{(\sigma+(\nu-1)/2)k}\,\frac{d^{m}\hat{P}_{k}(0)}{dx^{m}}=\frac{(-1)^{m}\, m!\, q^{m\sigma+m(m-1)/2}}{(q^{\sigma};q){}_{m+1}\,(q^{\sigma+\nu};q){}_{m+1}}\,.\label{eq:sum_diff_Phat0}
\end{equation}
\end{lemma}

\begin{proof} For a given $m\in\mathbb{N}$, one derives from (\ref{eq:eigeval_recurr_Phat})
the three-term inhomogeneous recurrence relation
\begin{equation}
q^{(\nu-1)/2}\,\frac{d^{m}\hat{P}_{n+1}(0)}{dx^{m}}-(1+q^{\nu})\,\frac{d^{m}\hat{P}_{n}(0)}{dx^{m}}+q^{(\nu+1)/2}\,\,\frac{d^{m}\hat{P}_{n-1}(0)}{dx^{m}}=-mq^{n}\,\frac{d^{m-1}\hat{P}_{n}(0)}{dx^{m-1}},\text{ }n\geq0,\label{eq:diff_Phat0_recurr}
\end{equation}
with the initial conditions
\begin{equation}
\frac{d^{m}\hat{P}_{-1}(0)}{dx^{m}}=0,\text{ }\frac{d^{m}\hat{P}_{0}(0)}{dx^{m}}=\delta_{m,0}\ \text{ }\text{for}\ \text{all}\text{ }m\geq0.\label{eq:diff_Phat0_ini}
\end{equation}
Recall that, by Proposition~\ref{prop:Phi}, the sequence $\{q^{(\nu-1)n/2}\hat{P}_{n}(x)\}$
converges on $\mathbb{C}$ locally uniformly and hence it is locally
uniformly bounded. Combining this observation with Cauchy's integral
formula one justifies that, for any $m\in\mathbb{Z}_{+}$ fixed, the
sequence $\{q^{(\nu-1)n/2}d^{m}\hat{P}_{n}(0)/dx^{m}\}$ is bounded
as well. Therefore the LHS of (\ref{eq:sum_diff_Phat0}) is well defined.
Let us call it $S_{m,\sigma}$. Applying summation in $n$ to (\ref{eq:diff_Phat0_recurr})
and bearing in mind (\ref{eq:diff_Phat0_ini}) one derives the recurrence
\[
S_{m,\sigma}=-\frac{mq^{\sigma}}{(1-q^{\sigma})\,(1-q^{\sigma+\nu})}\, S_{m-1,\sigma+1}\ \text{ }\text{for}\text{ }m\geq1,\sigma>0.
\]
Particularly, for $m=0$ we know that $\hat{P}_{n}(0)=Q_{n}^{(1)}$,
$n\in\mathbb{Z}_{+}$. Whence
\[
\forall\sigma>0,\ S_{0,\sigma}=\frac{1}{(1-q^{\sigma})\,(1-q^{\sigma+\nu})}
\]
(cf. (\ref{eq:sols_Q1_Q2})). A routine application of mathematical
induction in $m$ proves (\ref{eq:sum_diff_Phat0}). \end{proof}

\begin{proof}[Proof of Proposition~\ref{prop:char_fceT}] Letting
$\sigma=1$ in (\ref{eq:sum_diff_Phat0}) and making use of the locally
uniform convergence (cf. Proposition~\ref{prop:Phi}) one has
\[
\frac{1}{m!}\frac{d^{m}}{dx^{m}}\sum_{k=0}^{\infty}q^{(\nu+1)k/2}\hat{P}_{k}(x)\Big|_{x=0}=\frac{(-1)^{m}q^{m(m+1)/2}}{(q;q)_{m+1}\,(q^{\nu+1};q){}_{m+1}}\,,\text{ }\forall m\in\mathbb{Z}_{+}.
\]
Now, since $\Phi(x)$ is analytic it suffices to refer to formula
(\ref{eq:Phi}) to obtain
\[
\Phi(x)=\frac{1}{1-q^{\nu}}\,\sum_{n=0}^{\infty}\,\frac{(-1)^{n}q^{(n-1)n/2}\, x^{n}}{(q;q)_{n}\,(q^{\nu+1};q)_{n}}=\frac{1}{1-q^{\nu}}\,\,_{1}\phi_{1}(0;q^{\nu+1};q,x).
\]
Letting $\sigma=1-\nu$ in (\ref{eq:sum_diff_Phat0}), a fully analogous
computation can be carried out to evaluate the RHS of (\ref{eq:Psi})
thus getting formula (\ref{eq:Psi-qBessel}) for $\Psi(x)$.

From (\ref{eq:Hahn-Exton_diffeq}) it is seen that the sequences $\{u_{k}(x);\, k\in\mathbb{Z}\}$
and $\{v_{k}(x);\, k\in\mathbb{Z}\}$, where
\[
u_{k}(x)=q^{k/2}\, J_{\nu}(q^{k/2}\sqrt{x};q)\ \text{and}\ v_{k}(x)=q^{k/2}\, J_{-\nu}(q^{(k-\nu)/2}\sqrt{x};q),
\]
obey both the difference equation
\begin{equation}
\alpha_{k}u_{k+1}+\beta_{k}u_{k}+\alpha_{k-1}u_{k-1}=xu_{k}\label{eq:diff_eq}
\end{equation}
(with $\alpha_{k}$, $\beta_{k}$ being defined in (\ref{eq:alpha_beta})).
In the case of the former sequence, $\nu$ can be arbitrary positive,
and in the case of the latter one we assume that $0<\nu<1$. Hence
the sequence $\left(u_{0}(x),u_{1}(x),u_{2}(x),\ldots\right)$ is
a formal eigenvector of the Jacobi matrix $\mathcal{T}$ if and only
if $u_{-1}(x)=0$. A similar observation holds true if we replace
$u_{k}(x)$ by $u_{k}(\kappa,x)$. In view of Proposition~\ref{prop:char_fceT},
it suffices to notice that $u_{-1}(x)$ is proportional to $\Phi(x)$
and $u_{-1}(\kappa,x)$ to $\kappa\Phi(x)+\Psi(x)$. \end{proof}

\subsection{The case $\nu=0$ \label{subsec:nu_eq_0}}

The case $\nu=0$ is very much the same thing as the case when $0<\nu<1$.
First of all, this is again an indeterminate case, i.e. $T_{\text{min}}$
is not self-adjoint. On the other hand, there are some differences
causing the necessity to modify several formulas, some of them rather
substantially. Perhaps the main reason for this is the fact that the
characteristic polynomial of the difference equation with constant
coefficients, (\ref{eq:const_coeff}), has one double root if $\nu=0$
while it has two different roots if $0<\nu$. Here we summarize the
basic modifications but without going into details since the arguing
remains quite analogous.

For $\nu=0$ one has $\Dom\mathfrak{t}=\{f\in\ell^{2};\,\mathcal{A}f\in\ell^{2}\}$,
and two distinguished solutions of (\ref{eq:three-term_Z}) are
\[
Q_{n}^{(1)}=(n+1)q^{n/2},\ Q_{n}^{(2)}=q^{n/2},\ n\in\mathbb{Z},
\]
where again $Q_{n}^{(1)}=\hat{P}_{n}(0)$ for $n\geq0$ and $\{Q_{n}^{(2)}\}$
is a minimal solution, $W_{n}(Q^{(1)},Q^{(2)})=1$. The asymptotic
expansion of a sequence $f\in\Dom T_{\text{max}}$ reads
\[
f_{n}=\big(C_{1}\,(n+1)+C_{2}\big)q^{n/2}+o(q^{n})\,\ \text{as}\ n\to\infty,
\]
with $C_{1},C_{2}\in\mathbb{C}$. The one-parameter family of self-adjoint
extensions of $T_{\text{min}}$ is again denoted $T(\kappa)$, $\kappa\in P^{1}(\mathbb{R})$.
Definition (\ref{eq:Dom_Tkappa}) of $\Dom T(\kappa)$ formally remains
the same but the constants $C_{1}(f)$, $C_{2}(f)$ in the definition
are now determined by the limits
\[
C_{1}(f)=\lim_{n\to\infty}f_{n}\,(n+1)^{-1}q^{-n/2}\,,\ C_{2}(f)=\lim_{n\to\infty}\big(f_{n}-C_{1}(f)\,(n+1)q^{n/2}\big)q^{-n/2}\,.
\]
One still has $T(\infty)=T^{\text{F}}$. Similarly, $f\in\Dom T_{\text{max}}$
belongs to $\Dom T_{\text{min}}$ if and only if $C_{1}(f)=C_{2}(f)=0$
meaning that (\ref{eq:Dom_Tmin}) is true for $\nu=0$, too. Furthermore,
everything what is claimed in Propositions~\ref{prop:G_eq_Tinv}
and \ref{prop:T_pp_semibound} about the values $0<\nu<1$ is true
for $\nu=0$ as well.

Proposition~\ref{prop:form_Tkappa} should be modified so that the
quadratic form associated with a self-adjoint extension $T(\kappa)$,
$\kappa\in\mathbb{R}$, is $\mathfrak{t}_{\kappa+1}$, i.e. for $\nu=0$
one lets $\tau\equiv\tau(\kappa)=\kappa+1$. On the other hand, Proposition~\ref{prop:xi_n}
holds verbatim true also for $\nu=0$.

Relation (\ref{eq:2nd_recurr_Phat}) is valid for $\nu=0$ as well
but more substantial modifications are needed in Proposition~\ref{prop:Phi}.
One has
\[
\frac{q^{-n/2}}{n+1}\,\hat{P}_{n}(x)\to\Phi(x)=1-x\sum_{k=0}^{\infty}q^{k/2}\hat{P}_{k}(x)\,\ \text{as}\ n\to\infty,
\]
and the convergence is locally uniform on $\mathbb{C}$ for one can
estimate
\[
\left|\frac{q^{-n/2}}{n+1}\,\hat{P}_{n}(x)\right|\leq\prod_{k=0}^{n}\big(1+(k+1)q^{k}|x|\big),\ n\in\mathbb{Z}_{+}.
\]
Equation (\ref{eq:Phat_Phi}) should be replaced by
\[
\hat{P}_{n}(x)=\Phi(x)Q_{n}^{(1)}+xQ_{n}^{(2)}\sum_{k=0}^{n}Q_{k}^{(1)}\hat{P}_{k}(x)+xQ_{n}^{(1)}\sum_{k=n+1}^{\infty}Q_{k}^{(2)}\hat{P}_{k}(x).
\]
From here one infers the asymptotic expansion
\[
\hat{P}_{n}(x)=\Phi(x)(n+1)q^{n/2}+xq^{n/2}\sum_{k=0}^{\infty}Q_{k}^{(1)}\hat{P}_{k}(x)+o(q^{n})\,\ \text{as}\ n\to\infty.
\]
One concludes that what is claimed in Proposition~\ref{prop:specT_Phi}
about the values $0<\nu<1$ is true for $\nu=0$ as well but instead
of (\ref{eq:Psi}) one should write
\[
\Psi(x)=-x\sum_{k=0}^{\infty}(k+1)q^{k/2}\hat{P}_{k}(x).
\]

Finally let us consider modifications needed in Proposition~\ref{prop:char_fceT}.
For $\nu=0$ one has
\[
\Phi(x)=\,_{1}\phi_{1}(0;q;q,x)=J_{0}(q^{-1/2}\sqrt{x};q)
\]
and
\begin{eqnarray*}
\Psi(x) & = & \frac{\partial}{\partial p}\,\,_{2}\phi_{2}(0,q;pq,pq;q,px)\bigg|_{p=1}\\
 & = & 2q\,\frac{\partial}{\partial p}\,\,_{1}\phi_{1}(0;p;q,x)\bigg|_{p=q}+x\,\frac{\partial}{\partial x}\,\,_{1}\phi_{1}(0;q;q,x).
\end{eqnarray*}
Let
\[
u_{k}(x)=q^{(k+1)/2}J_{0}(q^{-k/2}\sqrt{x};q)
\]
and
\begin{eqnarray*}
v_{k}(x) & = & (k+1)q^{(k+1)/2}\,_{1}\phi_{1}(0;q;q,q^{k+1}x)+2q^{(k+3)/2}\,\frac{\partial}{\partial p}\,\,_{1}\phi_{1}(0;p;q,q^{k+1}x)\bigg|_{p=q}\\
 &  & +\, q^{(k+1)/2}x\,\frac{\partial}{\partial x}\,\,_{1}\phi_{1}(0;q;q,q^{k+1}x),
\end{eqnarray*}
$k\in\mathbb{Z}$. Then both sequences $\{u_{k}(x)\}$ and $\{v_{k}(x)\}$
solve (\ref{eq:diff_eq}) on $\mathbb{Z}$ and $u_{-1}(x)=\Phi(x)$,
$v_{-1}(x)=\Psi(x)$. Consequently, if $\Phi(x)=0$ then components
of an eigenvector of $T(\infty)=T^{\text{F}}$ corresponding to the
eigenvalue $x$ can be chosen to be $u_{k}(x)$, $k\in\mathbb{Z}{}_{+}$.
Similarly, if $\kappa\Phi(x)+\Psi(x)=0$ for some $\kappa\in\mathbb{R}$
then components of an eigenvector of $T(\kappa)$ corresponding to
the eigenvalue $x$ can be chosen to be $\kappa u_{k}(x)+v_{k}(x)$,
$k\in\mathbb{Z}{}_{+}$.

\section{Some applications to the $q$-Bessel functions}

In this section we are going to only consider the Friedrichs extension
if $0<\nu<1$. To simplify the formulations below we will unify the
notation and use the same symbol $T^{\text{F}}$ for the corresponding
self-adjoint Jacobi operator for all values of $\nu>0$, this is to
say even in the case when $\nu\geq1$. Making use of the close relationship
between the spectral data for $T^{\text{F}}$ and the $q$-Bessel
functions, as asserted in Propositions~\ref{prop:specT_Phi} and
\ref{prop:char_fceT}, we are able to reproduce in an alternative
way some results from \cite{KoelinkSwarttouw,Koelink}.

\begin{proposition}[Koelink, Swarttouw] Assume that $\nu>0$. The
zeros of $z\mapsto J_{\nu}(z;q)$ are all real (arranged symmetrically
with respect to the origin), simple and form an infinite countable
set with no finite accumulation points. Let $0<w_{1}<w_{2}<w_{3}<\ldots$
be the positive zeros of $J_{\nu}(z;q)$. Then the sequences
\begin{equation}
u(n)=\big(J_{\nu}(q^{1/2}w_{n};q),q^{1/2}J_{\nu}(qw_{n};q),qJ_{\nu}(q^{3/2}w_{n};q),\ldots\big),\ n\in\mathbb{N},\label{eq:OG_basis}
\end{equation}
form an orthogonal basis in $\ell^{2}$. In particular, the orthogonality
relation
\begin{equation}
\sum_{k=0}^{\infty}q^{k}J_{\nu}(q^{(k+1)/2}w_{m};q)\, J_{\nu}(q^{(k+1)/2}w_{n};q)=-\frac{q^{-1+\nu/2}}{2w_{n}}\, J_{\nu}(q^{1/2}w_{n};q)\,\frac{\partial J_{\nu}(w_{n};q)}{\partial z}\,\delta_{m,n}\label{eq:qJ_ortho}
\end{equation}
holds for all $m,n\in\mathbb{N}$.

\end{proposition}

\begin{remark*} It is not difficult to show that the proposition
remains valid also for $-1<\nu\leq0$. To this end, one can extend
the values $\nu>0$ to $\nu=0$ following the lines sketched in Subsection~\ref{subsec:nu_eq_0}
and employ Propositions~\ref{prop:specT_Phi} and \ref{prop:char_fceT}
while letting $\kappa=0$ in order to treat the values $-1<\nu<0$.
But we omit the details. An original proof of this proposition can
be found in \cite[Section~3]{KoelinkSwarttouw}. \end{remark*}

\begin{proof} All claims, except the simplicity of zeros and the
normalization of eigenvectors, follow from the known spectral properties
of $T^{\text{F}}$. Namely, $T^{\text{F}}$ is positive definite,
$(T^{\text{F}})^{-1}$ is compact, $\spec T^{\text{F}}=\{qw_{n}^{\,2};\, n\in\mathbb{N}\}$
and corresponding eigenvectors are given by formula (\ref{eq:egvector_TF});
cf. Propositions~\ref{prop:G_eq_Tinv}, \ref{prop:T_pp_semibound},
\ref{prop:specT_Phi} and \ref{prop:char_fceT}.

The remaining properties can be derived, in an entirely standard way,
with the aid of discrete Green's formula. Suppose a sequence of differentiable
functions $u_{n}(x)$, $n\in\mathbb{Z}$, obeys the difference equation
(\ref{eq:diff_eq}). Then Green's formula implies that, for all $m,n\in\mathbb{Z}$,
$m\leq n$,
\[
\sum_{k=m}^{n}u_{k}(x)^{2}=\alpha_{m-1}\big(u_{m-1}'(x)u_{m}(x)-u_{m-1}(x)u_{m}'(x)\big)-\alpha_{n}\big(u_{n}'(x)u_{n+1}(x)-u_{n}(x)u_{n+1}'(x)\big)
\]
(with the dash standing for a derivative). We choose $m=0$ and $u_{k}(x)$
as defined in (\ref{eq:egvector_TF}). From definition (\ref{eq:def_q-Bessel})
one immediately infers the asymptotic behavior
\begin{equation}
u_{k}(x)=C(x)(1+O(q^{k}))\, q^{(\nu+1)k/2},\ u_{k}'(x)=C'(x)(1+O(q^{k}))\, q^{(\nu+1)k/2},\ \text{as}\ k\to\infty,\label{eq:asympt_uk}
\end{equation}
where $C(x)=x^{\nu/2}\,(q^{1+\nu};q)_{\infty}/(q;q)_{\infty}$. It
follows that one can send $n\to\infty$ in Green's formula. For $x=qw_{n}^{\,2}$
we have $u_{-1}(x)=0$ and the formula reduces to the equality
\[
\sum_{k=0}^{\infty}q^{k}J_{\nu}(q^{(k+1)/2}w_{n};q)^{2}=-q^{\nu/2}\, J_{\nu}(q^{1/2}w_{n};q)\,\frac{\partial J_{\nu}(q^{-1/2}\sqrt{x};q)}{\partial x}\bigg|_{x=qw_{n}^{\,2}}\,.
\]
Whence (\ref{eq:qJ_ortho}). From the asymptotic behavior (\ref{eq:asympt_uk})
it is also obvious that $u_{k}(x)\neq0$ for sufficiently large $k$.
Necessarily, $\partial J_{\nu}(w_{n};q)/\partial z\neq0$. \end{proof}

In addition, one obtains at once an orthogonality relation for the
sequence of orthogonal polynomials $\{\hat{P}_{n}(x)\}$. As is well
known from the general theory \cite{Akhiezer} and Proposition~\ref{prop:sa_extns_Tmin},
the orthogonality relation is unique if $\nu\geq1$ and indeterminate
if $0<\nu<1$. It was originally derived in \cite[Theorem~3.6]{Koelink}.

\begin{proposition}[Koelink] Assume that $\nu>0$ and let $\{\hat{P}_{n}(x)\}$
be the sequence of orthogonal polynomials defined in (\ref{eq:Pn_def}),
(\ref{eq:Phat_def}), and $0<w_{1}<w_{2}<w_{3}<\ldots$ be the positive
zeros of $z\mapsto J_{\nu}(z;q)$. Then the orthogonality relation
\begin{equation}
-2q^{1-\nu/2}\,\sum_{k=1}^{\infty}\frac{w_{k}J_{\nu}(q^{1/2}w_{k};q)}{\partial J_{\nu}(w_{k};q)/\partial z}\,\hat{P}_{m}(qw_{k}^{\,2})\hat{P}_{n}(qw_{k}^{\,2})=\delta_{m,n}\label{eq:ogrel_polys}
\end{equation}
holds for all $m,n\in\mathbb{Z}_{+}$. \end{proposition}

\begin{proof} Let $u(k)$, $k\in\mathbb{N}$, be the orthogonal basis
in $\ell^{2}$ introduced in (\ref{eq:OG_basis}), i.e. we put
\[
u(k)_{n}=q^{n/2}J_{\nu}(q^{(n+1)/2}w_{k};q),\ k\in\mathbb{N},\, n\in\mathbb{Z}_{+}.
\]
Notice that the norm $\|u(k)\|$ is known from (\ref{eq:qJ_ortho}).
The vectors $u(k)$ and $\hat{P}(x)=(\hat{P}_{0}(x),\hat{P}_{1}(x),\hat{P}_{2}(x),\ldots)$,
with $x=qw_{k}^{\,2}$, are both eigenvectors of $T^{\text{F}}$ corresponding
to the same eigenvalue. Hence these vectors are linearly dependent
and one has
\[
q^{n/2}J_{\nu}(q^{(n+1)/2}w_{k};q)=J_{\nu}(q^{1/2}w_{k};q)\hat{P}_{n}(qw_{k}^{\,2}),\ k\in\mathbb{N},\, n\in\mathbb{Z}_{+}.
\]
One concludes that Parseval's identity
\[
\sum_{k=1}^{\infty}\frac{u(k)_{m}u(k)_{n}}{\|u(k)\|^{2}}=\delta_{m,n},\ m,n\in\mathbb{Z}_{+},
\]
yields (\ref{eq:ogrel_polys}). \end{proof}

\begin{remark} To complete the picture let us mention two more results
which are known about the Hahn-Exton $q$-Bessel functions and the
associated polynomials. First, denote again by $w_{n}^{(\nu)}\equiv w_{n}$,
$n\in\mathbb{N}$, the increasingly ordered positive zeros of $J_{\nu}(z;q)$.
In \cite{AbreuBustozCardoso} it is proved that if $q$ is sufficiently
small, more precisely, if $q^{\nu+1}<(1-q)^{2}$ then
\[
q^{-m/2}>w_{m}>q^{-m/2}\!\left(1-\frac{q^{m+\nu}}{1-q^{m}}\right)\!,\text{ }\forall m\in\mathbb{N}.
\]
More generally, in Theorem~2.2 and Remark~2.3 in \cite{AnnabyMansour}
it is shown that for any $q$, $0<q<1$, one has
\[
w_{m}=q^{-m/2}\left(1+O(q^{m})\right)\text{ }\text{as}\text{ }m\to\infty.
\]

Second, in \cite{KoelinkSwarttouw,Koelink} one can find an explicit
expression for the sequence of orthogonal polynomials $\{\hat{P}_{n}(x)\}$,
namely
\[
\hat{P}_{n}(x)=q^{n/2}\,\sum_{j=0}^{n}\frac{q^{n(j-\nu/2)}(q^{-n};q)_{j}}{(q;q)_{j}}\,\,_{2}\phi_{1}(q^{j-n},q^{j+1};q^{-n};q,q^{-j+\nu})\, x^{j},\ n\in\mathbb{Z}_{+}.
\]
Let us remark that a relative formula in terms of the Al-Salam--Chihara
polynomials has been derived in \cite[Theorem~2]{VanAssche}. \end{remark}

\section*{Acknowledgments}

The authors wish to acknowledge gratefully partial support from grant
No. GA13-11058S of the Czech Science Foundation.

\end{document}